\documentclass{amsart}
\usepackage{amssymb}
\usepackage{amsfonts}
\usepackage{amssymb}
\usepackage{amsmath}
\usepackage{amsthm}
\usepackage{enumerate}
\usepackage{tabularx}
\usepackage{url}
\usepackage{centernot}
\usepackage{mathtools}
\usepackage{stmaryrd}
\usepackage{amsthm,amssymb}
\usepackage{etoolbox}
\usepackage{tikz}
\usepackage{amssymb}
\usepackage{tablefootnote}
\usetikzlibrary{matrix}
\usepackage{graphics,graphicx}
\usepackage[all]{xy}
\usepackage{tikz-cd}
\usepackage{tikz}
\usepackage{marginnote}
\definecolor{mygray}{gray}{0.85}
\usepackage[linecolor=black,backgroundcolor=mygray,colorinlistoftodos,prependcaption,textsize=small]{todonotes}
\usepackage{xargs}                      

\renewcommand{\leq}{\leqslant}
\renewcommand{\geq}{\geqslant}

\makeatletter
\def\subsection{\@startsection{subsection}{3}%
  \z@{.5\linespacing\@plus.7\linespacing}{.3\linespacing}%
  {\bfseries\centering}}
\makeatother

\makeatletter
\def\subsubsection{\@startsection{subsubsection}{3}%
  \z@{.5\linespacing\@plus.7\linespacing}{.3\linespacing}%
  {\centering}}
\makeatother

\makeatletter
\def\myfnt{\ifx\protect\@typeset@protect\expandafter\footnote\else\expandafter\@gobble\fi}
\makeatother

\newtheorem{theorem}{Theorem}[section]

\newtheorem{corollary}[theorem]{Corollary}
\newtheorem{definition}[theorem]{Definition}
\newtheorem{lemma}[theorem]{Lemma}
\newtheorem{proposition}[theorem]{Proposition}
\newtheorem{example}[theorem]{Example}

\newtheorem{fact}[theorem]{Fact}
\newtheorem{conclusion}[theorem]{Conclusion}
\newtheorem{remark}[theorem]{Remark}
\newtheorem{notation}[theorem]{Notation}
\newtheorem{convention}[theorem]{Convention}

\newtheorem{definition/fact}[theorem]{Definition/Fact}

\newtheorem*{theorem1.2}{Theorem~1.2}

\newcounter{claimcounter}

\newcommand{\pureindep}[1][]{%
  \mathrel{
    \mathop{
      \vcenter{
        \hbox{\oalign{\noalign{\kern-.3ex}\hfil$\vert$\hfil\cr
              \noalign{\kern-.7ex}
              $\smile$\cr\noalign{\kern-.3ex}}}
      }
    }\displaylimits_{#1}
  }
}

\newcommand{\indep}[2]{%
  \mathrel{
    \mathop{
      \vcenter{
        \hbox{%
\oalign{
\noalign{\kern-.3ex}\hfil$\vert$\hfil\cr
              \noalign{\kern-.7ex}
              $\smile$\cr\noalign{\kern-.3ex}
}
}
      }
}^{\!\!\!\!\!#2}_{\!\!\hspace{-0.1em}#1}
  }
}

\newcommand{\dindep}[1][]{\indep{#1}{d}}

\begin{document}

\begin{abstract} We use a variation on Mason's $\alpha$-function as a pre-dimension function to construct a not one-based $\omega$-stable plane $P$ (i.e. a simple rank $3$ matroid) which does not admit an algebraic representation (in the sense of matroid theory) over any field. Furthermore, we characterize forking in $Th(P)$, we prove that algebraic closure and intrinsic closure coincide in $Th(P)$, and we show that $Th(P)$ fails weak elimination of imaginaries, and has Morley rank~$\omega$.
\end{abstract}

\title{A New $\omega$-Stable Plane}
\thanks{The author would like to thank heartfully John Baldwin for helpful discussions related to this paper. The present paper was written while the author was a post-doc research fellow at the Einstein Institute of Mathematics of the Hebrew University of Jerusalem, supported by European Research Council grant 338821.}

\author{Gianluca Paolini}
\address{Department of Mathematics ``Giuseppe Peano'', University of Torino, Via Carlo Alberto 10, 10123, Italy.}
\email{gianluca.paolini@unito.it}

\date{\today}
\maketitle


\section{Introduction}


	In this study we use methods from combinatorial theory and model theory to construct  a simple rank $3$ matroid that is new from both perspectives. 
%
%
%
As well-known to experts, the class of simple rank $3$ matroids corresponds canonically to the class of linear spaces, or, equivalently, to the class of geometric lattices of rank $3$.
	Matroid theorists refer to simple rank $3$ matroid also as {\em planes}, and so we will adopt this terminology in this study. In \cite{geometric_lattices} we used Crapo's theory of one point-extensions of matroids \cite{crapo} to construct examples of $\omega$-stable (one of the most important dividing lines in model theory) planes in the context of abstract elementary classes. In the present study we use Mason's $\alpha$-function of matroid theory and the amalgamation construction known as Hrushovski's construction to build an $\omega$-stable plane in the context of classical first-order logic with an interesting combination of combinatorial and model theoretic properties. 
	
	Mason's $\alpha$-function is a naturally arising notion of complexity for matroids introduced by Mason \cite{mason} in his study of so-called gammoids, a now well-known class of matroids arising from paths in graphs. Interestingly, Evans recently showed \cite{evans} that the class of strict gammoids corresponds exactly to the class of finite geometries considered by Hrushovski in his celebrated refutation of Zilber's conjecture \cite{udi}. A model theoretic analysis of Mason's $\alpha$-function similar to our approach but quite different in motivation has also appeared in \cite{omer2, omer1}\footnote{Our study and \cite{omer2, omer1} were concurrent and both inspired by Evans's work \cite{evans}.}.
	
	What is referred to as Hrushovski's constructions is a method of constructing model theoretically well-behaved structures via an amalgamation procedure which makes essential use of a certian predimension function. This amalgamation construction results in a countable structure (the so-called ``Hrushovski's generic'') carrying the additional structure of an infinite dimensional matroid, which controls important model theoretic properties of the structure constructed (which in our case is a simple rank $3$ matroid). In this type of constructions, the specifics of the predimension function depend on the case at hand, and in our case the predimension function is a mild but crucial variation of Mason's $\alpha$ function (cf. Definition~\ref{def_delta}).
	
	We believe that our variation on Mason's $\alpha$-function as a predimension function is of independent interest from various points of view. On the combinatorial side, we think that the combinatorial consequences of the realization that Mason's $\alpha$-function is essentially a predimension  should be properly explored. On the model theoretic side, we believe that the ``collapsed\footnote{The so-called ``collapse'' is a technical variation of the Hrushovki's construction which ensures the satisfaction of further important model-theoretic properties, as e.g. uncountable categoricity: only one model up to isomorphism in every uncountable cardinalily.} version'' of our construction leads to interesting connections with {\em design theory} (currently explored in works in preparation joint with John Baldwin \cite{paper_with_John, John_paper}), and in particular with Steiner $k$-systems. 
	
	Before stating our main theorem we spend few motivating words introducing the model theoretic properties appearing in it. Model theory is the study of classes of structures using logical properties, which are also referred to as dividing lines (since they are often accompanied by a dichotomous behavior). Among the various properties considered by logicians there are certain properties which are more ``geometric'' in nature. These properties are called ``geometric'' since they are given imposing conditions on certain infinite dimensional matroids associated with the structures. 
	The canonical example of this kind of structures is the class of {\em strongly minimal structures}, where the model theoretic operator of algebraic closure determines an infinite dimensional matroid. This context has later been extended to uncountably categorical structures (one model up to isomorphism in every uncountable cardinality), and even more generally to $\omega$-stable structures (cf. \cite[Chapter~6]{marker}).
	In this spirit, one of the geometric properties of the kind mentioned above is the notion of being {\em one-based}, which on strongly minimal structures corresponds to the natural notion of local modularity of the lattice of closed sets of the associated matroid.
	
\smallskip
		
	We prove the following theorem:


	\begin{theorem}\label{second_main_th} There exists a pre-dimension function $\delta$ on the class of finite planes (finite simple rank $3$ matroids) such that the corresponding ``Hrushovski's generic'' (cf. Definition~\ref{defgen}) exists, and so it is a plane $P$ (i.e. a simple rank $3$ matroid, cf.~Definition~\ref{def_matroid_intro}), and it satisfies the following conditions:
	\begin{enumerate}[(1)]
	\item $P$ contains the ``non-Desarguesian'' matroid (cf.~Figure~\ref{figure1}, or \cite[pg. 139]{welsh}), and so it is not algebraic (in the sense of matroid theory);
	\item in $Th(P)$ intrinsic closure and algebraic closure coincide (cf. Definition \ref{geom_clo});
	\item\label{weak_elim} $Th(P)$ does {\em not} have weak elimination of imaginaries (cf. Definition~\ref{def_weak_elim});
	\item $Th(P)$ is not one-based (cf. Definition~\ref{def_one_based});
	\item $Th(P)$ is $\omega$-stable and has Morley rank $\omega$ (cf. \cite[Chapter~6]{marker});
	\item over algebraically closed sets forking in $Th(P)$ corresponds to the canonical amalgamation introduced in \cite[Theorem 4.2]{geometric_lattices} (cf. Remark \ref{remark_canonical_amalgam}).
	\end{enumerate}	
\end{theorem}	

	As mentioned above, properties (2)-(6) of Theorem~\ref{second_main_th} are important dividing lines in model theory, and their satisfaction shows that our object is particularly well-behaved from this perspective; for an introduction to these notions see e.g. \cite[Chapter~6]{marker}. 
	In combination with these properties, the fact that our plane $P$ is not algebraic makes our plane particularly exotic. Non-algebraic planes are somewhat rare in nature, and in fact the existence of non-algebraic planes is a non-trivial fact due to Lindstr\"om \cite{lind, lind_alg}, who constructed in \cite{lind} an infinite familiy of non-algebraic finite planes. Furthermore, this shows that our variation on Mason's $\alpha$-function is crucial, since the class of finite simple matroids of positive $\alpha$ is the already mentioned class of strict gammoids, and these structures are known to be linear (see e.g. \cite[Corollary~7.75]{aigner}), and thus in particular algebraic. On the other hand, in \cite{hom_matroid} we constructed a simple rank $3$ matroid with strong homogeneity properties with $\wedge$-embeds all the finite simple rank $3$ matroids, and so in particular it is {\em not} algebraic, but that structure has the so-called {\em independence property}, and so it is in a completely different region of the model theoretic universe. In fact, we stress once again that what is interesting about the structure constructed in this paper is the combination of the failure of algebraicity together with the satisfaction of $\omega$-stability, and of the other model theoretic properties of Theorem~\ref{second_main_th}.

	

	
	Concerning the structure of the paper: in Section~\ref{preliminaries} we give a quick introduction to matroid theory and recall the definition of Mason's $\alpha$-function; 
in Section~\ref{constr_sec} we introduce the construction at the core of this paper and prove Theorem~\ref{second_main_th}.

	Finally, we stress that in our construction the fact that we use a modification of the $\alpha$-function (and not the $\alpha$-function per se) is crucial, since the $\delta$ which we define in Section~\ref{constr_sec} is submodular, while the $\alpha$-function is not submodular, cf. Remark~\ref{remark_alpha}.

\section{The Framework}\label{preliminaries}

	For a thorough introduction to matroids see e.g. \cite{rota} or \cite{aigner}. Referring to the formalism of matroid theory which uses as primitive the notion of {\em dependent set}, a simple rank $3$ matroid can be defined as the following combinatorial structure:

\begin{definition}\label{def_matroid_intro} A {\em simple matroid of rank $\leq 3$} is a $3$-hypergraph $(V, R)$ whose adjacency relation is irreflexive, symmetric \mbox{and satisfies the following {\em exchange axiom}:}
	\begin{enumerate}[(Ax)]
	\item if $R(a, b, c)$ and $R(a, b, d)$, then $\{a, b, c, d \}$ is an $R$-clique.
\end{enumerate}
\end{definition}

	\begin{remark}\label{remark_def_matroid} Clearly Definition~\ref{def_matroid_intro} is formally not a complete definition, since it does not specify the dependent sets of cardinality different than $3$. The point is that if $M$ is a simple matroid of rank $\leq 3$, then: \begin{enumerate}[(1)]
	\item every set of size $< 3$ is not dependent;
	\item every set of size $> 3$ is dependent.
\end{enumerate}
Hence, every structure as in Definition~\ref{def_matroid_intro} admits canonically the structure of a matroid of rank $\leq 3$ in the sense of any matroid theory textbook.
\end{remark}

	\begin{notation} Given a simple matroid $M$ we denote by $\mathrm{cl_M}$ the canonically associated closure operator. Thus, we will denote matroids also as $(M, \mathrm{cl_M})$. Further, we will denote as $G(M)$ the geometric lattice canonically associated to $M$, i.e. the collections of closed sets (sets of the form $\mathrm{cl}_M(X) = X$) of $M$ under inclusion.
\end{notation}

	We will also need the following definition (which is used in Fact~\ref{notation_canonical_amalgam}).

	\begin{definition}\label{def_wedge} Let $M = (M, cl)$ and $N = (N, cl)$ be simple matroids. We say that $M$ is a $\wedge$-subgeometry of $N$ if $M$ is a subgeometry of $N$ (i.e. $M \subseteq N$ and $cl_M(X) = cl_N(X) \cap M$) and the inclusion map $\mathit{i}_M : M \rightarrow N$ induces an embedding (with respect to both $\vee$ and $\wedge$) of $G(M)$ into $G(N)$ (cf. \cite[Section 2]{geometric_lattices}).
\end{definition}
	
%
%
	
	
	\begin{notation}\label{notation_flat} Let $M = (M, cl)$ be a simple matroid.
	\begin{enumerate}[(1)]
	\item We refer to closed subsets of $M$ (i.e. subsets $F \subseteq M$ of the form $cl_M(F) = F$) as {\em flats} of $M$, or $M$-flats. 
	\item Given two subsets $F$ and $X$ of $M$ we use the notation\footnote{This notation is taken from \cite{mason} where the notion of $\alpha$-function was introduced.} $F \preccurlyeq X$ (resp. $F \prec X$) to mean that $F$ is a subset of $X$ (resp. a proper subset) and $F$ is a flat of $M$.
\end{enumerate} 
\end{notation}

	\begin{definition}[Mason's $\alpha$-function \cite{mason}]\label{def_alpha_function} Let $M$ be a finite simple matroid. For each subset $X$ of $M$ we define recursively:
	$$\alpha(X) = |X| - rk(X) - \sum_{F \prec X} \alpha(F).$$
\end{definition}

	\begin{definition}\label{nullity} Let $M$ be a finite simple matroid and $F$ an $M$-flat. We define the nullity of $F$ as follows:
	$$\mathbf{n}(F) = |F| - rk(F).$$
\end{definition}

	The following conventions will simplify a great deal the computations of Section~\ref{constr_sec}. Its use will be limited to Proposition \ref{comp_prop} and Lemma \ref{delta_lemma}.

	\begin{convention}\label{convention} Let $M = (M, cl)$ and $N = (N, cl)$ be finite simple matroids and suppose that $M$ is a subgeometry of $N$. If $F$ is an $N$-flat, then:
	\begin{enumerate}[(1)]
	\item we denote by $|F|_M$ the number $|F \cap M|$;
	\item we denote by $\mathbf{n}_M(F)$ the number $\mathbf{n}(F \cap M)$ computed in $M$ as an $M$-flat, and by $\mathbf{n}_N(F)$ the number $\mathbf{n}(F)$ computed in $N$ as an $N$-flat.
\end{enumerate}	 
\end{convention}

	\begin{convention}\label{convention_lines} Let $M = (M, cl)$ be a simple matroid. Then:
	\begin{enumerate}[(1)]
	\item $M$-flats of rank $2$ are referred to as {\em lines}; 
	\item we denote by $L(M)$ the set of lines of $M$.
	\item For $N \subseteq M$ and $\ell \in L(M)$, we say that $\ell$ is based in $N$ if $|\ell \cap N| \geq 2$.
	\item For $N \subseteq M$, we let $L_M(N)$ to be the set of $\ell \in L(M)$ which are based in $N$. Since $L(N)$ and $L_M(N)$ are in canonical bijection we will be sloppy in distinguishing between them, and often write $L(N)$ instead of $L_M(N)$.
\end{enumerate}	
\end{convention}

	The following remark gives an explicit characterization of $\alpha(M)$ in the case $M$ is of rank $3$. For the purposes of the present paper this characterization suffices, and thus we could have avoided the general definition of the $\alpha$-function; we chose not to do so because we wanted to motivate the naturality of the predimension function of Definition~\ref{def_delta} (from Section~\ref{constr_sec}) and make explicit its relation to the $\alpha$-function.

	\begin{remark}\label{null_remark} Let $M$ be a finite simple matroid of rank $3$, then:
	$$\alpha(M) = |M| - 3 - \sum_{\ell \in L(M)} \mathbf{n}(\ell).$$
\end{remark}


\section{The Construction}\label{constr_sec}

	We follow the general framework of \cite{baldwin_generic}, and refer to proofs from there when minor changes to the arguments are needed in order to establish our claims.

	\begin{notation}\label{notation} Let $\mathbf{K}^{*}_0$ be the class of finite simple matroids of rank $\leq 3$ seen as structures in a language with a ternary predicate $R$ for dependent sets of size $3$ (cf. Definition~\ref{def_matroid_intro} and Remark~\ref{remark_def_matroid}). Recall that we refer to elements in $A \in \mathbf{K}^{*}_0$ as {\em planes} (if $rk(A) < 3$ we say that $A$ is degenerate).
\end{notation}

	\begin{convention}\label{conv} Throughout the rest of the paper model theoretically we will consider our planes only in the language of Notation~\ref{notation}. In particular, if $P$ is a plane seen as an $L$-structure, then the lines of $P$ (in the sense of the associated geometric lattice $G(P)$) are not elements of $P$, but only definable subsets of $P$.
\end{convention}
	
	\begin{definition}\label{def_delta} For $A \in \mathbf{K}^{*}_0$, let: 
	$$\delta(A) = |A| - \sum_{\ell \in L(A)} \mathbf{n}(\ell).$$
\end{definition}

	\begin{remark}\label{remark_alpha} Notice that, by Remark \ref{null_remark}, if $A \in \mathbf{K}^{*}_0$ has rank $3$, then:
	$$\delta(A) = \alpha(A) + 3.$$
And so our $\delta$ is just a natural variation of Mason's $\alpha$-function. Despite this, our variation is crucial, since, as we observed in the introduction, the class of finite simple matroids of positive $\alpha$ is the already mentioned class of strict gammoids, and these structures are known to be linear (see e.g. \cite[Corollary~7.75]{aigner}), while, as shown in Example~\ref{non_des_ex}, there exists a non-algebraic $A \in \mathbf{K}^{*}_0$ such that $\delta(A) \geq 0$.
\newline  Even more interestingly, although as a consequence of Lemma~\ref{delta_lemma}, $\delta$ is submodular, the $\alpha$-function is {\em not} submodular, i.e. there exists $A, B \in \mathbf{K}^{*}_0$ such that:
	$$\alpha(A \cup B) > \alpha(A) + \alpha(B) - \alpha(A \cap B).$$
In fact letting $A = \{a, b, c \}$ and $B = \{a, b, d \}$ be two copies of the three element simple matroid of rank $3$ we have that:
	$$\alpha(A \cup B) = 1 > \alpha(A) + \alpha(B) - \alpha(A \cap B) = 0 + 0 + 0.$$
\end{remark}
	

	\begin{example}\label{non_des_ex} Let $A$ be the ``non-Desarguesian'' matroid (cf.~Figure~\ref{figure1}, for another representation of this matroid see \cite[pg. 139]{welsh}). Then, $\delta(A) = 1$, since $A$ has $10$ points and exactly $9$ non-trivial lines, each of nullity $1$ (i.e. each has size $3$). Furthermore, inspection of Figure~\ref{non_des_ex} shows that for every $B \subseteq A$, we have that $\delta(B) \geq 0$. The ``non-Desarguesian'' matroid was shown not to be algebraic in \cite[Corollary, pg.~238]{lind_alg}. This will be relevant for the proof of Theorem~\ref{second_main_th}(1). Finally, notice on the other hand that $\alpha(A) < 0$ (where $\alpha$ is Mason's $\alpha$-function from Def.~\ref{def_alpha_function}), and so the class of planes with positive $\delta$ but negative $\alpha$ is non-trivial, as in fact all the matroids with non-negative $\alpha$ are linear (as they are gammoids).
\begin{figure}[htb]
\begin{center}
\includegraphics[clip, width=0.90\textwidth]{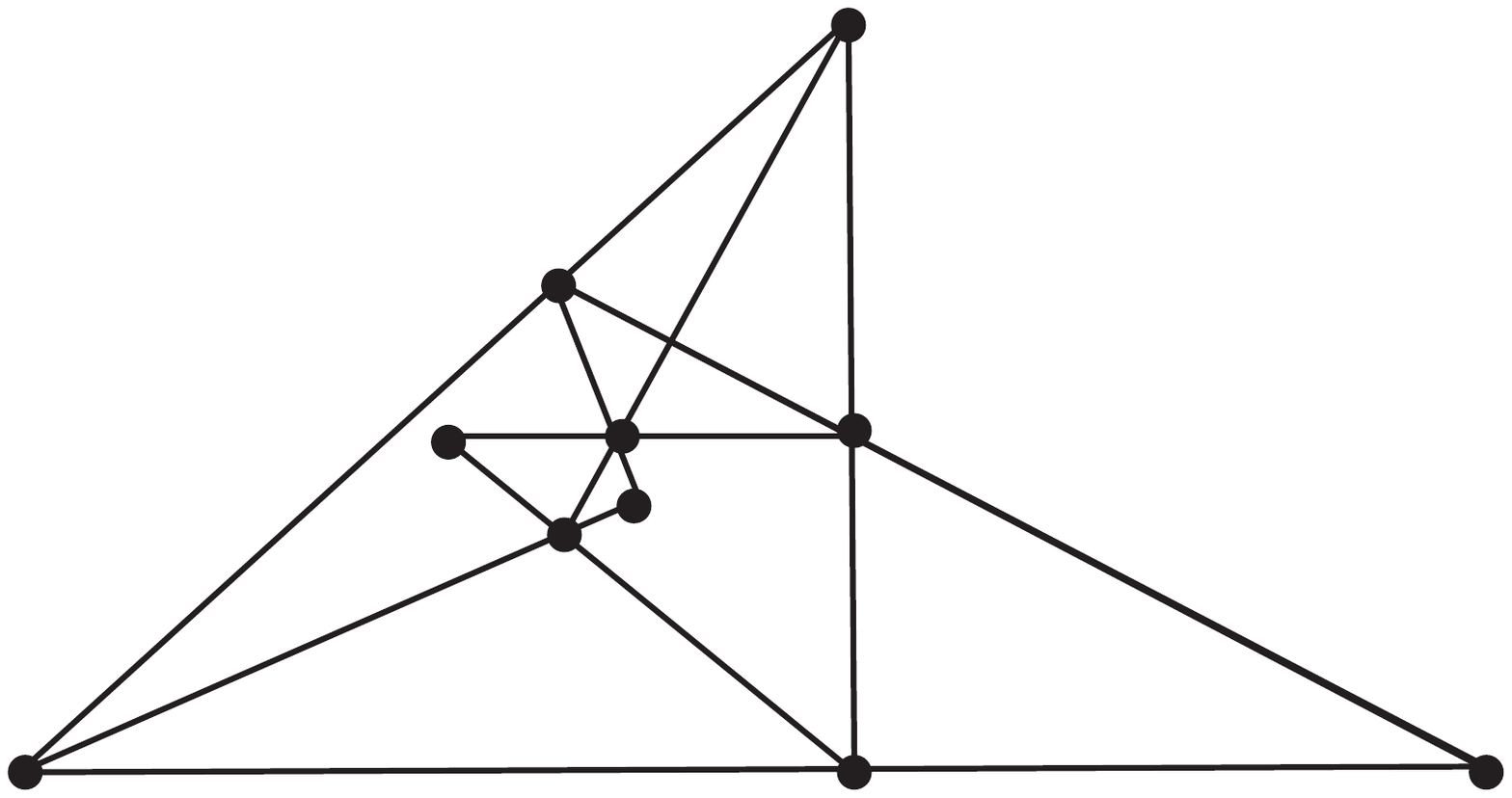}\caption{The ``non-Desarguesian'' matroid.}\label{figure1}
\end{center}
\end{figure}
\end{example}

	The following two claims constitutes the computational core of the paper, and aim at proving that our function $\delta$ is lower semimodular. Proposition \ref{comp_prop} is used to prove Lemma \ref{delta_lemma}, which in turn is used to draw Conclusion \ref{conclusion_ax}. In Proposition \ref{comp_prop} and Lemma \ref{delta_lemma} we will make a crucial \mbox{use of Conventions~\ref{convention}-\ref{convention_lines}.}

	\begin{proposition}\label{comp_prop} Let $A$ and $B$ be disjoint subsets of a matroid $C \in \mathbf{K}^{*}_0$ (so, in particular $A$ and $B$ are submatroids of the matroid $C$, or, equivalently, substructures in the sense of Notation~\ref{notation}). Then:
	\begin{enumerate}[(1)]
	\item if $\ell \in L(B)$, then $\mathbf{n}_{AB}(\ell) - \mathbf{n}_{B}(\ell) = |\ell|_A$ (clearly, if $\ell \in L(B)$, then $\ell \in L(AB)$);
	\item $\delta(A/B) := \delta(AB) - \delta(B)$ is equal to:
	$$|A| - \displaystyle \sum_{\substack{\ell \in L(AB) \\ \ell \in L(A) \\ \ell \not\in L(B)}} \mathbf{n}_{AB}(\ell) - \sum_{\substack{\ell \in L(AB) \\ \ell \in L(A) \\ \ell \in L(B)}} |\ell|_A - \sum_{\substack{\ell \in L(AB) \\ \ell \in L(B) \\ \ell \not\in L(A)}} |\ell|_A.$$
\end{enumerate}
\end{proposition}

	\begin{proof} Concerning item (1), for $\ell \in L(AB)$ and $\ell \in L(B)$ we have: 
\[ \begin{array}{rcl}
   \mathbf{n}_{AB}(\ell) - \mathbf{n}_{B}(\ell) & = & |\ell|_{AB} - rk(\ell) - |\ell|_B + rk(\ell) \\
				 & = & |\ell|_A + |\ell|_B - |\ell|_B \\
				 & = & |\ell|_A.
\end{array}	\]
Concerning item (2), we have that $\delta(A/B)$ is:
\[ \begin{array}{rcl}
   &  = & \displaystyle |AB| - \sum_{\ell \in L(AB)} \mathbf{n}_{AB}(\ell) - |B| + \sum_{\ell \in L(B)} \mathbf{n}_B(\ell) \\
   &  = & \displaystyle |A| +|B| - \sum_{\ell \in L(AB)} \mathbf{n}_{AB}(\ell) - |B| + \sum_{\ell \in L(B)} \mathbf{n}_B(\ell) \\
   &  = & \displaystyle |A| - \sum_{\ell \in L(AB)} \mathbf{n}_{AB}(\ell) + \sum_{\ell \in L(B)} \mathbf{n}_B(\ell) \\
   &  = & |A| - \displaystyle \sum_{\substack{\ell \in L(AB) \\ \ell \in L(A) \\ \ell \not\in L(B)}} \mathbf{n}_{AB}(\ell) - \sum_{\substack{\ell \in L(AB) \\ \ell \in L(A) \\ \ell \in L(B)}} \mathbf{n}_{AB}(\ell) - \sum_{\substack{\ell \in L(AB) \\ \ell \in L(B) \\ \ell \not\in L(A)}} \mathbf{n}_{AB}(\ell) + \displaystyle \sum_{\ell \in L(B)} \mathbf{n}_{B}(\ell) \\
				 & = & |A| - \displaystyle \sum_{\substack{\ell \in L(AB) \\ \ell \in L(A) \\ \ell \not\in L(B)}} \mathbf{n}_{AB}(\ell) - \sum_{\substack{\ell \in L(AB) \\ \ell \in L(A) \\ \ell \in L(B)}} (\mathbf{n}_{AB}(\ell) - \mathbf{n}_{B}(\ell)) - \sum_{\substack{\ell \in L(AB) \\ \ell \in L(B) \\ \ell \not\in L(A)}} (\mathbf{n}_{AB}(\ell) - \mathbf{n}_{B}(\ell)) \\
				 & = & |A| - \displaystyle \sum_{\substack{\ell \in L(AB) \\ \ell \in L(A) \\ \ell \not\in L(B)}} \mathbf{n}_{AB}(\ell) - \sum_{\substack{\ell \in L(AB) \\ \ell \in L(A) \\ \ell \in L(B)}} |\ell|_A - \sum_{\substack{\ell \in L(AB) \\ \ell \in L(B) \\ \ell \not\in L(A)}} |\ell|_A.
\end{array}	\]
Concerning the passage from the third equation to the fourth equation notice that if $\ell \in L(AB) - (L(A) \cup L(B))$, then $\mathbf{n}_{AB}(\ell) = 0$.	
\end{proof}

	\begin{lemma}\label{delta_lemma} Let $A, B, C \subseteq D \in \mathbf{K}^{*}_0$, with $A \cap C = \emptyset$ and $B \subseteq C$. Then: 
	$$\delta(A/B) \geq \delta(A/C).$$
\end{lemma}

	\begin{proof} Let $A, B, C$ be subsets of a matroid $D$ and suppose that $B \subseteq C$ and $A \cap C = \emptyset$. Notice that by Proposition \ref{comp_prop} we have:
\begin{equation}\label{equation1}\tag{$\star_1$}
- \delta(A/C) = - |A| + \displaystyle \sum_{\substack{\ell \in L(AC) \\ \ell \in L(A) \\ \ell \not\in L(C)}} \mathbf{n}_{AC}(\ell) + \sum_{\substack{\ell \in L(AC) \\ \ell \in L(A) \\ \ell \in L(C)}} |\ell|_A + \sum_{\substack{\ell \in L(AC) \\ \ell \in L(C) \\ \ell \not\in L(A)}} |\ell|_A,
\end{equation} 
\begin{equation}\label{equation2}\tag{$\star_2$}
\delta(A/B) = |A| - \displaystyle \sum_{\substack{\ell \in L(AB) \\ \ell \in L(A) \\ \ell \not\in L(B)}} \mathbf{n}_{AB}(\ell) - \sum_{\substack{\ell \in L(AB) \\ \ell \in L(A) \\ \ell \in L(B)}} |\ell|_A - \sum_{\substack{\ell \in L(AB) \\ \ell \in L(B) \\ \ell \not\in L(A)}} |\ell|_A.
\end{equation} 
Notice now that for $\ell \in L(AC)$ we have:
	\begin{enumerate}[(a)]
	\item if $\ell \in L(A)$, $\ell \notin L(B)$ and $\ell \notin L(C)$, then $\ell$ occurs in the first sum of (\ref{equation1}) and in the first sum of (\ref{equation2}), and clearly $\mathbf{n}_{AC}(\ell) \geq \mathbf{n}_{AB}(\ell)$;
	\item if $\ell \in L(A)$ and $\ell \in L(B)$, then $\ell \in L(A)$ and $\ell \in L(C)$, and so $\ell$ occurs in the second sum of (\ref{equation1}) and in the second sum of (\ref{equation2});
	\item if $\ell \in L(B)$ and $\ell \notin L(A)$, then $\ell \in L(C)$ and $\ell \notin L(A)$, and so $\ell$ occurs in the third sum of (\ref{equation1}) and in the third sum of (\ref{equation2});
	\item if $\ell \in L(A)$, $\ell \notin L(B)$ and $\ell \in L(C)$, then $\ell$ occurs in the second sum of (\ref{equation1}) and in the first sum of (\ref{equation2}), and furthermore we have: 
	$$\mathbf{n}_{AB}(\ell) \leq \mathbf{n}_{A}(\ell) + 1 < \mathbf{n}_{A}(\ell) + 2 = |\ell|_A.$$
\end{enumerate}
Since, clauses (a)-(d) above cover all the terms occurring in (\ref{equation2}), we conclude that $\delta(A/B) \geq \delta(A/C)$, as wanted.
\end{proof}


	\begin{definition} Let:
	$$\mathbf{K}_0 = \{ A \in \mathbf{K}^{*}_0 \text{ such that for any } A' \subseteq A,\delta(A') \geq 0\}, $$
and $(\mathbf{K}_0, \leq)$ be as in \cite[Definition 3.11]{baldwin_generic}, i.e. we let $A \leq B$ if and only if:
	$$A \subseteq B \wedge \forall X (A \subseteq X \subseteq B \Rightarrow \delta(X) \geq \delta(A)).$$
Finally, we write $A < B$ to mean that $A \leq B$ and $A$ is a proper subset of $B$.
\end{definition}

	\begin{conclusion}\label{conclusion_ax} $(\mathbf{K}_0, \leq)$ satisfies Axiom A1-A6 from \cite[Axioms Group A]{baldwin_generic}, i.e.:
\begin{enumerate}[(1)]
	\item if $A \in \mathbf{K}_0$, then $A \leq A$;
	\item if $A \leq B$, then $A \subseteq B$;
	\item if $A, B, C \in \mathbf{K}_0$ and $A \leq B \leq C$, then $A \leq C$;
	\item if $A, B, C \in \mathbf{K}_0$, $A \leq C$, $B \subseteq C$, and $A \subseteq B$, then $A \leq B$;
	\item $\emptyset \in \mathbf{K}_0$ and $\emptyset \leq A$, for all $A \in \mathbf{K}_0$;
	\item if $A, B, C \in \mathbf{K}_0$, $A \leq B$, and $C$ is a substructure of $B$, then $A \cap C \leq C$.
\end{enumerate}
\end{conclusion}

	\begin{proof} As in e.g. \cite[Theorem 3.12]{baldwin_generic}, this is easy to establish using Lemma \ref{delta_lemma}.
\end{proof}

	\begin{fact}[{\cite[Theorem 4.2]{geometric_lattices}}]\label{notation_canonical_amalgam} Let $A, B, C \in \mathbf{K}_0$ with $C$ a $\wedge$-subgeometry (cf. Definition \ref{def_wedge}) of $A$ and $B$ and $A \cap B = C$. Then there exists a canonical amalgam of $A$ and $B$ over $C$, which we denote as $A \oplus_C B$. In the next remark we give an explicit characterization of $A \oplus_C B$ as an $L$-structure, i.e.  we simply translate the lattice theoretic definition of $A \oplus_C B$ from \cite{geometric_lattices} into the language of $L$-structures.
\end{fact}

\begin{remark}\label{remark_canonical_amalgam} 
The amalgam $D := A \oplus_{C} B$ of Fact \ref{notation_canonical_amalgam} can be characterized as the following $L$-structure:
\begin{enumerate}[(1)]
	\item the domain of $D$ is $A \cup B$;
	\item $R^{D} = R^{A} \cup R^{B} \cup \{ \{ a, b, c \} : a \vee b \vee c = a' \vee b' \text{ and } \{ a', b' \} \subseteq C \}$.
\end{enumerate} 
Where $\vee$ refers to the canonically associated geometric lattice $G(D)$. A more transparent way to define the amalgam $A \oplus_{C} B$ is by defining the domain of $A \oplus_{C} B$ to be simply $A \cup B$, and the lines of $A \oplus_{C} B$ to be the lines coming from $A$, those coming from $B$, modulo identifying the lines from $C$, plus the obvious trivial lines. 
\end{remark}

	\begin{lemma}\label{amalgamation} 
	\begin{enumerate}[(1)]
	\item If $A \leq B \in \mathbf{K}_0$, then $A$ is a $\wedge$-subgeometry of $B$.
	\item $(\mathbf{K}_0, \leq)$ has the amalgamation property.
	\end{enumerate}
\end{lemma}

	\begin{proof} Concerning (1), suppose that $A,  B \in \mathbf{K}_0$, and $A$ is not a $\wedge$-subgeometry of $B$, then there exists $p \in B - A$ and $\ell_1 \neq \ell_2 \in L(A)$ such that $p$ is incident with both $\ell_1$ and $\ell_2$. Thus, $\delta(Ap) < \delta(A)$ and so $A \not\leq B$. Concerning (2), let $A, B, C \in \mathbf{K}_0$ and suppose that $C \leq A, B$ with $A \cap B = C$ (without loss of generality). Let $A \oplus_C B := D$ (recall Notation \ref{notation_canonical_amalgam}), which exists by (1). Using e.g. Remark \ref{remark_canonical_amalgam}, it is easy to see that:
\begin{equation}\tag{$\star_3$}
\delta(D) = \delta(A) + \delta(B) - \delta(C).
\end{equation}
Furthermore, for every $C \subseteq X \subseteq D$ we have that $X = (A \cap X) \oplus_{C \cap X} (B \cap X)$. Thus, it is immediate to infer that $D \in \mathbf{K}_0$ and $B, C \leq D$, as wanted.
\end{proof}

	We now introduce several technical notions of amalgamation, in particular sharp and uniform amalgamation. We are only interested in sharp amalgamation as a sufficient condition for uniform amalgamation, and we are only interested in the latter as a sufficient condition for $\omega$-stability, see Conclusion~\ref{conclusion_stability}.

	\begin{definition}\label{various_def} Let $(\mathbf{L}_0, \leq)$ be a class of relational structures of the same vocabulary satisfying the conditions in Conclusion~\ref{conclusion_ax} and $A, B, C \in \mathbf{L}_0$.
	\begin{enumerate}[(1)]
	\item For $k < \omega$, we say that $A$ is $k$-strong in $B$, denoted $A \leq^k B$, if for any $B'$ with $A \subseteq B' \subseteq B$ and $|B' - A| \leq k$ we have $A \leq B'$ (cf. \cite[Definition 2.26]{baldwin_generic}).
	\item We say that $B$ is a primitive extension of $A$ if $A \leq B$ and there is no $A \subsetneq B_0 \subsetneq B$ such that $A \leq B_0 \leq B$ (cf. \cite[Definition 2.30]{baldwin_generic}).
	\item Given $C \leq A, B$ with $A \cap B = C$, we let $A \otimes_C B$ denote the free amalgam of $A$ and $B$ over $C$, i.e. the structure with domain $A \cup B$ and no additional relations apart from the ones in $A$ and the ones in $B$.
	\item We say that $(\mathbf{L}_0, \leq)$ has the sharp amalgamation property if for every $A, B, C \in \mathbf{L}_0$, if $C \leq A$ is primitive and $C \leq^{|A| - |C|} B$, then either $A \otimes_C B \in \mathbf{L}_0$ or there is a $\leq$-embedding of $A$ into $B$ over $C$ (cf. \cite[Definition 2.31]{baldwin_generic}).
	\item\label{unif_amal} We say that $(\mathbf{L}_0, \leq)$ has the uniform amalgamation property if the following
condition holds: for every $A \leq B \in \mathbf{L}_0$, and for every $m < \omega$ there is an $n = f_B(m)$
such that if $A \leq^n C$, then there is a $D$, a strong embedding of $C$ into $D$ and an $m$-strong
embedding of $B$ into $D$ that completes a commutative diagram with the given embeddings of $A$ into $B$ and $C$. 
	\end{enumerate}
\end{definition}

	\begin{proposition}\label{prop_princ} Let $A \leq B \in \mathbf{K}_0$ be primitive. Then either $|B-A| \leq 1$, or for every $p \in B-A$ we have that $p$ is not incident with a line $\ell \in L(A)$. Furthermore, in the first case we have that $\delta(B/A) \leq 1$. 
\end{proposition}

	\begin{proof} Suppose that there exists $p \in B-A$ such that $p$ is incident with a line $\ell \in L(A)$ (and thus under no other line $\ell' \in L(A)$, cf. Lemma \ref{amalgamation}). Then we have $\delta(A) = \delta(Ap)$, and so if $|B| - |A| > 1$ we have $\delta(A) = \delta(Ap) \leq \delta(B)$, and thus $A < Ap < B$, contradicting the assumptions of the proposition. The furthermore part is immediate from the definition of the function $\delta$.
\end{proof}

	\begin{lemma}\label{sharp_amalgamation} \begin{enumerate}[(1)]
	\item\label{sharp} $(\mathbf{K}_0, \leq)$ has the sharp amalgamation property.
	\item\label{sharp_improv} In (\ref{sharp}) we can replace $|A|-|C|$ with $1$, i.e. the conclusion of Definition~\ref{various_def}(4) is true for the all the extensions of the form $C \leq^{1} B$, not only for the extensions of the form $C \leq^{|A| - |C|} B$, as required by Definition~\ref{various_def}(4)).
	\item\label{item_4} $(\mathbf{K}_0, \leq)$ has the uniform amalgamation property (cf. Definition~\ref{various_def}(\ref{unif_amal})).
	\end{enumerate}
\end{lemma}

	\begin{proof}
	Item (\ref{item_4}) follow from (\ref{sharp}) by \cite[Lemma 2.32]{baldwin_generic}. We prove (\ref{sharp}) and (\ref{sharp_improv}). Let $A, B, C \in \mathbf{K}_0$ and suppose that $C < A$ is primitive, $C \leq^1 B$ and $A \cap B = C$ (without loss of generality). By Proposition \ref{prop_princ}, either every $p \in A-C$ is not incident with a line $\ell \in L(C)$ or $C - A = \{ p \}$ and there exists a line $\ell \in L(C)$ such that $p$ is incident with $\ell$. Suppose the first, then by Remark \ref{remark_canonical_amalgam} the canonical amalgam $A \oplus_C B$ (cf. Notation \ref{notation_canonical_amalgam}) coincide with the free amalgam $A \otimes _C B$ (cf. Definition \ref{various_def}(3)), and so we are done. Suppose the second and let $p$ and $\ell$ witness it. If every $p' \in B -C$ is not incident with the line $\ell$, then also in this case $A \oplus_C B = A \otimes _C B$, and so we are done. Finally, if there exists $p' \in B - C$ such that $p$ is incident with $\ell$, then clearly $A = Cp$ is such that it $\leq$-embeds into $B$ over $C$, since $\delta(C) = \delta(Cp') = \delta(Cp)$.
\end{proof} 

	\begin{definition}\label{defgen} Let $(\mathbf{L}_0, \leq)$ be a class of relational structures in the language $L$ satisfying the conditions in Conclusion~\ref{conclusion_ax}.
A countable
$L$-model $M$ is said to be {\em $(\mathbf{L}_0, \leq)$-generic} when:
\begin{enumerate}[(1)]
\item
     if $A\leq M, A\leq B\in \mathbf{L}_0$, then there exists $B'\leq M $
such that $ B \cong_{A} B'$;
\item $M$ is a union of finite substructures.
\end{enumerate}
\end{definition}

	\begin{fact}[{\cite[Theorem 2.12]{baldwin_generic}}]\label{fact_ex_gen} Let $(\mathbf{L}_0, \leq)$ be a class of relational structures of the same vocabulary satisfying the conditions in Conclusion~\ref{conclusion_ax}, and suppose that $(\mathbf{L}_0, \leq)$ has the amalgamation property. Then there exists a $(\mathbf{L}_0, \leq)$-generic model, and this model is unique up to isomorphism.
\end{fact}

	\begin{corollary}\label{cor_ex} The $(\mathbf{K}_0, \leq)$-generic model exists.
\end{corollary}

	\begin{proof} By Fact~\ref{fact_ex_gen} and Lemma~\ref{amalgamation}.
\end{proof}

	\begin{notation} 
	\begin{enumerate}[(1)]
	\item Let $P$ be the generic model for $(\mathbf{K}_0, \leq)$ (cf. Corollary~\ref{cor_ex}), and let $\mathfrak{M}$ be the monster model of $Th(P)$.
	\item Given $A, B, C \subseteq \mathfrak{M}$ we write $A \equiv_C B$ to mean that there is an automorphism of $\mathfrak{M}$ fixing $C$ pointwise and mapping $A$ to $B$.
\end{enumerate}	  
\end{notation}

	We recall that we write $A \subseteq_{\omega} B$ to mean that $A \subseteq B$ and $|A| < \aleph_0$.

	\begin{definition}\label{def_indep} Let $M \models Th(P)$. 
	\begin{enumerate}[(1)]
	\item Given $A \subseteq_{\omega} M$, we let:
	$$d(A) = inf\{ \delta(B) : A \subseteq B \subseteq_{\omega} M \}.$$
	\item Given $A \subseteq_{\omega} M$, we let $A \leq M$ if $d(A) = \delta(A)$.
	\item\label{item_indep} Given $A, B, C \subseteq_{\omega} M$ with $C \leq A, B \leq M$ and $A \cap B = C$, we let $A \dindep[C] B$ if:
	$$d(A/C) = d(A/B).$$
\end{enumerate}	  
\end{definition}

	\begin{fact}\label{closure_fact} Let $M \models Th(P)$ and $A \subseteq_{\omega} M$. Then there exists a unique finite $B_A \subseteq_{\omega} M$ such that $A \subseteq B_A \leq M$ and $B_A$ is minimal with respect to inclusion. Furthermore, $B_A \subseteq acl_{M}(A)$ (where $acl_{M}(A)$ is the algebraic closure of $A$ in~$M$).
\end{fact}

	\begin{proof} By \cite[Theorem 2.23]{baldwin_generic}, since clearly $\mathbf{K}_0$ has finite closure.
\end{proof}

	\begin{definition}\label{geom_clo} Following \cite{baldwin_generic} we denote the set $B_A$ from Fact \ref{closure_fact} by $icl_M(A)$, and we call it the intrinsic closure of $A$ in $M$.
\end{definition}
	
	\begin{lemma}\label{lemma_acl} Let $A \subseteq_{\omega} P$. Then $acl_{P}(A) \subseteq icl_{P}(A)$.
\end{lemma}

	\begin{proof} Let $A \subseteq_{\omega} P$, $b \in P - icl_{P}(A)$, $A' = icl_{P}(A)$ and $B' = icl_{P}(Ab)$. Now, for every $1 < k < \omega$, we can find $D \leq P$ such that:
	$$D \cong_{A'} \underbrace{B' \oplus_{A'} B' \oplus_{A'} \cdots \oplus_{A'} B'}_{\text{$k$-times}} := F,$$
since $A' \leq B' \leq F \in \mathbf{K}_0$ and $P$ is generic (cf. \cite[Definition 2.11]{baldwin_generic}).
Thus, by the homogeneity of $P$, we can find infinitely many elements of $P$ with the same type as $b$ over $A'$. Hence, $b \notin acl_{P}(A)$.
\end{proof}


	\begin{conclusion}\label{conclusion_geom_clo} Let $A \subseteq_{\omega} M \models Th(P)$, then $icl_{M}(A) = acl_{M}(A)$, i.e. intrinsic closure and algebraic closure coincide in $M$.
\end{conclusion}

	\begin{proof} The inclusion $icl_{M}(A) \subseteq acl_{M}(A)$ is by Fact \ref{closure_fact}. For the other inclusion argue as in \cite[Theorem 4.5]{baldwin_generic} using Lemma \ref{lemma_acl}.
\end{proof}

	\begin{proposition}\label{AB_closed_prop} Let $A, B, C \subseteq_{\omega} \mathfrak{M}$ with $C \leq A, B \leq \mathfrak{M}$ and $A \cap B = C$. If $A \dindep[C] B$ (cf. Definition~\ref{def_indep}(\ref{item_indep})), then $AB \leq \mathfrak{M}$.
\end{proposition}

	\begin{proof} As in \cite[Theorem 3.31]{baldwin_generic}.
\end{proof}

	\begin{lemma}\label{indep_lemma} Let $A, B, C \subseteq_{\omega} \mathfrak{M}$ with $C \leq A, B \leq \mathfrak{M}$ and $A \cap B = C$. Then the following are equivalent:
	\begin{enumerate}[(1)]
	\item $A \dindep[C] B$ (cf. Definition~\ref{def_indep}(\ref{item_indep}));
	\item $AB = A \oplus_{C} B$ (cf. Notation \ref{notation_canonical_amalgam}).
\end{enumerate}
\end{lemma}

	\begin{proof} Easy to see using Proposition \ref{AB_closed_prop} and Remark \ref{remark_canonical_amalgam}.
\end{proof}

	\begin{lemma}\label{stationarity_lemma} Let $A, B, C \subseteq_{\omega} \mathfrak{M}$ with $C \leq A, B \leq \mathfrak{M}$ and $A \cap B = C$. Then:
	\begin{enumerate}[(1)]
	\item (Existence) there exists $A' \equiv_C A$ such that $A' \dindep[C] B$;
	\item (Stationarity) $A \equiv_C A'$, $A \dindep[C] B$ and $A' \dindep[C] B$, then $A \equiv_{B} A'$.
\end{enumerate}	
\end{lemma}

	\begin{proof} Immediate from Lemma \ref{indep_lemma} and Remark \ref{remark_canonical_amalgam}.
\end{proof}

	\begin{conclusion}\label{conclusion_stability} $P$ is $\omega$-stable.
\end{conclusion}

	\begin{proof} As observed in Fact~\ref{closure_fact}, the class $\mathbf{K} = Mod(Th(P))$ has  finite closures. Thus, the result follows from Lemma \ref{sharp_amalgamation}, \cite[Theorem 2.28]{baldwin_generic}, \cite[Theorem 2.21]{baldwin_generic},  \cite[remark right after 2.20]{baldwin_generic} and \cite[Theorem 3.34]{baldwin_generic}, where the argument in \cite[Theorem 3.34]{baldwin_generic} goes through by Lemma \ref{stationarity_lemma}.
\end{proof}

	\begin{corollary}\label{char_forking} Let $A, B, C \subseteq_{\omega} \mathfrak{M}$ with $C \leq A, B \leq \mathfrak{M}$ and $A \cap B = C$. Then the following are equivalent:
	\begin{enumerate}[(1)]
	\item $A \pureindep[C] B$ (in the forking sense, cf. e.g. \cite[Chapter~6]{marker});
	\item $A \dindep[C] B$ (cf. Definition~\ref{def_indep}(\ref{item_indep}));
	\item $AB = A \oplus_{C} B$ (cf. Notation \ref{notation_canonical_amalgam}).
	\end{enumerate}
\end{corollary}

	\begin{proof} The equivalence (1) $\Leftrightarrow$ (2) is as in \cite[Lemma 3.38]{baldwin_generic} using Lemma \ref{stationarity_lemma}, the equivalence (2) $\Leftrightarrow$ (3) is Lemma \ref{indep_lemma}.
\end{proof}

%

	\begin{definition}[{\cite[Exercise~8.4.2]{{tent+ziegler}}}]\label{def_weak_elim} Let $T$ be a first-order theory. We say that $T$ has weak elimination of imaginaries if for every model $M \models T$ and definable set $D$ over $A \subseteq M$ there is a smallest algebraically closed set over which $D$ is definable.
\end{definition}

	\begin{corollary}\label{weak_elimi} $Th(P)$ does {\em not} have weak elimination of imaginaries (Def.~\ref{def_weak_elim}).
\end{corollary}

	\begin{proof} Let $\{ a, b, a', b' \} \subseteq \mathfrak{M}$ be such $|\{ a, b, a', b' \}| = 4$, $\{ a, b, a', b' \} \leq \mathfrak{M}$ and $\{a, b, a', b' \}$ forms an $R$-clique (i.e. the points $a, b, a', b'$ are collinear). Consider now the definable set $X = \{ a, b \} \cup \{ c \in \mathfrak{M} : \mathfrak{M} \models R(a, b, c) \}$ in $\mathfrak{M}$. Then in $\mathfrak{M}$ there is no smallest algebraically closed set over which $X$ is definable, since clearly $X = \{ a', b' \} \cup \{ c \in \mathfrak{M} : \mathfrak{M} \models R(a, b, c) \}$ and both $\{ a, b \}$ and $\{ a', b' \}$ are algebraically closed in $\mathfrak{M}$ (recall Conclusion \ref{conclusion_geom_clo}).
\end{proof}


	We now introduce the notion of a theory being one-based, a crucial property in geometric model theory.

	\begin{definition}\label{def_one_based} Let $T$ be an $\omega$-stable first-order theory, and let $\mathfrak{M}$ be its monster model. We say that $T$ is one-based if for every $A, B \subseteq \mathfrak{M}$ such that $A = acl(A)$ and $B = acl(B)$ we have that $A \pureindep[A \cap B] B$.
\end{definition}

\begin{proposition}\label{not_one_based} $Th(P)$ in {\em not} one-based.
\end{proposition}

	\begin{proof} Let $C \leq \mathfrak{M}$ be a simple rank $3$ matroid with domain $\{ p_1, p_2, p_3 \}$. Let $B \leq \mathfrak{M}$ be an extension of $C$ with a generic point $q_1$ (i.e. $q_1$ is not incident with any line from $C$). Let $D \leq \mathfrak{M}$ be an extension of $C$ with a new point $q_2$ under the line $p_1 \vee q_1$. Notice now that the the submatroid $A$ of $D$ with domain $\{ p_1, p_2, p_3, q_2 \}$ is such that $A \leq D$, since $\delta(A) = \delta(D)$. Thus, $A, B, C \leq \mathfrak{M}$, $A \cap B = C$ and $A \not\!\pureindep[C] B$ (by Corollary \ref{char_forking}).
\end{proof}



	The following four items are an adaptations of items 4.6, 4.8, 4.9, 4.10 of \cite{ziegler}. We will use them to show that $\mathfrak{M}$ has Morley rank $\omega$, using the argument laid out in \cite[Proposition 4.10]{ziegler}.
	

	\begin{lemma}\label{ziegler_lemma1} Let $B \leq C \in \mathbf{K}_0$ be a primitive extension (cf. Definition \ref{various_def}(2)). Then there are two cases:
	\begin{enumerate}[(1)]
	\item $\delta(C/B) = 1$ and $C = B \cup \{ c \}$;
	\item $\delta(C/B) = 0$.
	\end{enumerate}
\end{lemma}

\begin{proof} Suppose that $B \leq C \in \mathbf{K}_0$, $\delta(C/B) > 0$ and $c_1 \neq  c_2 \in C - B$.
We make a case distinction:
\newline {\em Case 1}. $c_1$ or $c_2$ is not incident with any line from $B$.
\newline Without loss of generality $c_1$ is not incident with any line from $B$. Then, $\delta(Bc_1) = \delta(B) + 1 \leq \delta(C)$, where the second inequality is because $\delta(C/B) > 0$, and so $B < Bc_1 < C$. Hence, in this case we have that $B \leq C$ is {\em not} primitive.
\newline {\em Case 2}. $c_1$ and $c_2$ are both incident with a line from $B$.
\newline Then $\delta(B) = \delta(Bc_1) \leq \delta(C)$ and so $B < Bc_1 < C$. Hence, also in this case we have that $B \leq C$ is {\em not} primitive.
\newline Thus, from the above argument we see that if $B \leq C$ is primitive and $\delta(C/B) > 0$, then $C = B \cup \{ c \}$, and so $\delta(C/B) = 1$ \mbox{(cf. Proposition \ref{prop_princ}).}
\end{proof}

	\begin{remark} Notice that it is possible that $B \leq C \in \mathbf{K}_0$ is primitive, $\delta(C/B) = 0$ and $|C - B| \geq 2$. To see this, consider the plane whose geometric lattice is represented in Figure \ref{myfigure2} and let $B = \{ a, b, c \}$ and $C = \{ a, b, c, d, e, f\}$.
\begin{figure}[ht]
	\begin{center}
		\begin{tikzpicture}
\matrix (a) [matrix of math nodes, column sep=0.4cm, row sep=0.3cm]{
   &  & & & & abcdef \\
ab & ac & bc & & adf & bd & cde & ae & bef & cf \\
a  &  b  &  c & & d & & e & & f \\
& & & & & \emptyset \\};
\end{tikzpicture}
\end{center}  \caption{An example.}\label{myfigure2}
\end{figure}
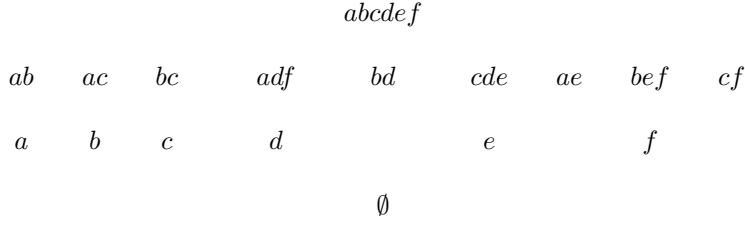
\end{remark}


	\begin{lemma}\label{ziegler_lemma2} Let $B \leq C \in \mathbf{K}_0$ be primitive, $C \leq \mathfrak{M}$, and suppose that $\delta(C/B) = 0$. Then $tp(C/B)$ is isolated and strongly minimal.
\end{lemma}

	\begin{proof} As in the proof of \cite[Lemma 4.8]{ziegler} replacing the free amalgam $A \otimes_C B$ with the canonical amalgam $A \oplus_C B$ (cf. Notation \ref{notation_canonical_amalgam}).
\end{proof}

	\begin{corollary}\label{ziegler_cor} Let $B \leq C \in \mathbf{K}_0$, $C \leq \mathfrak{M}$, and suppose that $\delta(C/B) = 0$. Then:
	\begin{enumerate}[(1)]
	\item $tp(C/B)$ has finite Morley rank;
	\item the Morley rank of $tp(C/B)$ is at least the length of a decomposition of $C/B$ into primitive extensions.
	\end{enumerate}
\end{corollary}

	\begin{proof} Exactly as in \cite[Corollary 4.9]{ziegler}.
\end{proof}

	\begin{proposition}\label{prop_for_ziegler} There exists finite $B \leq \mathfrak{M}$ and elements $q_k$, for $k < \omega$, such that $d(q_k/B) = 0$, and the extension $cl(Bq_k)$ has decomposition length $k$.
\end{proposition}

	\begin{proof} Let $B \leq \mathfrak{M}$ be a simple rank $3$ matroid with domain $\{ p_1, p_2, p_3 \}$. By induction on $k < \omega$, we define $B \leq Q_k \leq \mathfrak{M}$ such that $q_k \in Q_k$ . For $k = 0$, let $Q_0 \leq \mathfrak{M}$ be an extension of $B$ with a new point $q_0$ under the line $p_1 \vee p_2$. For $k = m+1$, let $Q_k \leq \mathfrak{M}$ be an extension of $Q_m$ with a new point $q_k$ under the line $p_2 \vee q_m$ if $m$ is even, and under the line $p_1 \vee q_m$ if $m$ is odd. Then clearly $d(q_k/B) = 0$, and the extension $cl(Bq_k) = Q_k$ has decomposition length $k$.
\end{proof}

	We now restate our main theorem and prove it.

	\begin{theorem1.2} There exists a pre-dimension function $\delta$ on the class of finite planes (finite simple rank $3$ matroids) such that the corresponding ``Hrushovski's generic'' (cf. Definition~\ref{defgen}) exists, and so it is a plane $P$ (i.e. a simple rank $3$ matroid, cf.~Definition~\ref{def_matroid_intro}), and it satisfies the following conditions:
	\begin{enumerate}[(1)]
	\item $P$ contains the ``non-Desarguesian'' matroid (cf.~Figure~\ref{figure1}, or \cite[pg. 139]{welsh}), and so it is not algebraic (in the sense of matroid theory, cf.~Definition~\ref{def_proj});
	\item in $Th(P)$ intrinsic closure and algebraic closure coincide (cf. Definition \ref{geom_clo});
	\item\label{weak_elim} $Th(P)$ does {\em not} have weak elimination of imaginaries (cf. Definition~\ref{def_weak_elim});
	\item $Th(P)$ is not one-based (cf. Definition~\ref{def_one_based});
	\item $Th(P)$ is $\omega$-stable and has Morley rank $\omega$ (cf. \cite[Chapter~6]{marker});
	\item over algebraically closed sets forking in $Th(P)$ corresponds to the canonical amalgamation introduced in \cite[Theorem 4.2]{geometric_lattices} (cf. Remark \ref{remark_canonical_amalgam}).
	\end{enumerate}	
\end{theorem1.2}	

	\begin{proof} Concerning item (1), notice that if a matroid is algebraic, then so is any of its submatroids. Thus, $P$ is not algebraic since it contains the ``non-Desarguesian'' matroid from Example~\ref{non_des_ex}, which is explicitly shown not to be algebraic in \cite[Corollary, pg. 238]{lind_alg}. Item (2) is Conclusion \ref{conclusion_geom_clo}. Item (3) is Corollary \ref{weak_elimi}. Item (4) is by Proposition \ref{not_one_based} and Conclusion \ref{conclusion_geom_clo}. Concerning item (5), argue as in \cite[Proposition 4.10]{ziegler} using Corollary \ref{ziegler_cor} and Proposition \ref{prop_for_ziegler}. Item (6) is by Corollary~\ref{char_forking} and Conclusion~\ref{conclusion_geom_clo}.
\end{proof}

\end{document}